\documentclass[12pt]{article}
\usepackage{amsmath,amssymb,amsfonts,amsthm}
\usepackage{eucal}%
\usepackage{enumerate}
\newcommand{\C}{\mathbb{C}}

\newcommand{\Z}{\mathbb{Z}}
\newcommand{\cO}{\mathcal{O}}

\newcommand{\cV}{\mathcal{V}}

\newcommand{\bW}{\mathsf{W}}
\newcommand{\cL}{\mathcal{L}}
\newcommand{\cI}{\mathcal{I}}

\newcommand{\bE}{\mathsf{E}}

\newcommand{\cG}{\mathcal{G}}
\newcommand{\cU}{\mathcal{U}}
\newcommand{\fZ}{\mathfrak{Z}}
\newcommand{\cH}{\mathcal{H}}

\newcommand{\Pt}{\mathbb{P}^2}

\newcommand{\cF}{\mathcal{F}}

\newcommand{\F}{\mathcal{Fl}}
\newcommand{\lt}{\mathfrak{t}}

\newcommand{\Gr}{\mathcal{G}r}
\newcommand{\V}{\mathcal{V}}

\DeclareMathOperator{\Hilb}{Hilb}
\DeclareMathOperator{\Ext}{Ext}

\DeclareMathOperator{\ch}{ch}
\DeclareMathOperator{\td}{td}

\DeclareMathOperator{\Lie}{Lie}
\DeclareMathOperator{\End}{End}

\DeclareMathOperator{\Hom}{Hom}
\newcommand{\cHom}{Hom}
\DeclareMathOperator{\supp}{supp}
\DeclareMathOperator{\spn}{span}

\newtheorem{thm}{Theorem}
\newtheorem{lemma}{Lemma}

\begin{document}
\title{Vertex Operators, Grassmannians, and Hilbert Schemes}
\author{Erik Carlsson}
\maketitle
 
\begin{abstract}
We approximate the infinite Grassmannian by finite-dimensional cutoffs, and define a family of
fermionic vertex operators as the limit of geometric correspondences on the equivariant cohomology
groups, with respect to a one-dimensional torus action. We prove that in the localization basis, 
these are the well-known fermionic vertex operators on the infinite wedge representation. 
Furthermore, the boson-fermion correspondence,
locality, and intertwining properties with the Virasoro algebra are the limits of relations on the 
finite-dimensional cutoff spaces, which are true for geometric reasons.
We then show that these operators are also, almost by definition,
the vertex operators defined by Okounkov and the author in \cite{CO},
on the equivariant cohomology groups of the Hilbert scheme of points on $\C^2$, with respect to
a special torus action.

\end{abstract}
 
\section{Introduction}
The infinite wedge representation,
\[\Lambda^{\infty/2}(F \cdot \Z) = \bigoplus_{m\in \Z} \Lambda^{\infty/2}_m (F \cdot \Z),\]
is the infinite-dimensional vector space over $F$ with basis
\begin{equation}
\label{wedgebasis}
e_{\mu,m} = e_{m+\mu_1} \wedge e_{m+\mu_2-1} \wedge e_{m+\mu_3-2} \wedge \cdots,
\end{equation}
where $e_j,j\in \Z$ is a basis of $F\cdot \Z$, and $\mu$ is a partition.
It realizes representations of the Virasoro algebra,
affine Ka\c{c}-Moody algebras, and other infinite-dimensional Lie algebras \cite{K,BO}.
These representations, and their relations with each other,
can be explicitly constructed by formal generating functions 
of operators known as vertex operators. These operators include
include the operators of ``wedging and contracting'' with $e_j$,
and a larger family with interesting intertwining properties with the above Lie algebras.

Surprisingly, these operators also appear in the cohomology groups of
Hilbert schemes of points on a smooth complex surface, and more general moduli spaces of sheaves.
If $S$ is a smooth surface, let $\Hilb_k S$ denote the 
the Hilbert scheme of zero-dimensional subschemes $Z \subset S$, such that $\dim(H^0(\cO_Z)) = k$.
There are many interesting and valuable correspondences acting on the direct sum
\begin{equation}
\label{F}
\cU'_S = \bigoplus_{k\geq 0} \cU_{S,k}', \quad \cU_{S,k}' = H^*(\Hilb_k S,\C),
\end{equation}
and when $S=\C^2$,
$\cU'_S \cong \Lambda^{\infty/2}_0(\C\cdot \Z)$, and we obtain geometric constructions of the above Lie algebras.

For instance, consider the Nakajima subvariety \cite{L,Nak1,NakL},
\begin{equation}
\begin{split}
Z_n = \big\{(Z,x,Z')  \in & \Hilb_k S \times S \times  \Hilb_{k+n} S\ \big| \\
 & Z \subset Z',\ \supp(\cO_{Z'}/\cO_{Z}) = \{x\} \big\},
\end{split}
\end{equation}
which is a singular variety with a well-defined fundamental class
\[[Z_n] \in H_*(\Hilb_k S) \otimes S \otimes H_*(\Hilb_{k+n} S).\]
If $S$ is a compact variety, then by Poincar\'{e} duality, we obtain an operator
\begin{equation}
\label{corr}
\alpha_n(x) : \cU_k' \rightarrow \cU_{k-n}',\quad
\alpha_n(x)(\gamma) = p_{3*}\left(\left(p_1^*(\gamma) \cup p_2^*(x)\right) \cap [Z_{-n}]\right)
\end{equation}
for $n < 0$, and dual operators for $n>0$, under the inner product
\begin{equation}
\int x \cup y.
\label{iprod}
\end{equation}
Nakajima proved that these operators
satisfy the commutation relations of an infinite-dimensional Heisenberg algebra.

Also using correspondences, Grojnowski \cite{Groj} described a family of vertex operators
on the cohomology groups. These
comprise a larger family of operators, but they are formally related to Nakajima's
operators. In the case of the Hilbert scheme on $\C^2$,
this relationship is just the boson-fermion correspondence \cite{FB,NakL},
and so one might say that the two collections of operators are equivalent. 
In a similar direction, Lehn constructed a
corresponding action of the Virasoro algebra using characteristic classes \cite{L}.

In \cite{CO}, Okounkov and the author defined a new family of vertex operators 
in terms of characteristic classes of universal bundles on the Hilbert scheme,
and proved a ``bosonization'' formula in terms of the Nakajima operators.
These operators were new, but in special cases, such as \eqref{T} below, they
agree with the vertex operators of the last paragraph. 
Furthermore, if $S$ carries a torus action, then
characteristic classes have an explicit expression in the fixed-point basis of equivariant
cohomology. In fact, the primary motivation was to calculate certain equivariant integrals
of characteristic classes on the Hilbert scheme of $S = \C^2$, with respect to a torus action \cite{NO,ErikTh}.
In this situation, $\Hilb_k S$ is not proper, but
one may simply define integration by the localization formula for equivariant cohomology.

In this paper, we define a family of vertex operators as limits of correspondences on
the equivariant cohomology groups of certain finite-dimensional Grassmannians.
These are the Grassmannians of half-dimensional subspaces of
the vector space of formal Laurent series with cutoffs in high and low degree (essentially the infinite Sato Grassmannian).
The circle acts on this space by rotation of functions.
We define operators in the same way as in \eqref{corr}, where
the correspondence is the two-step flag variety, and $x$ is replaced by a characteristic class.
We use the relationship between localization and geometry deduce
their (super) commutation relations, the boson-fermion correspondence, and their intertwining properties
with the Virasoro Lie algebra. We then conclude that a subset of these operators coincide the fermionic vertex operators
discussed, for instance, in \cite{K}, or \cite{FB} chapter 5, by identifying the fixed-point basis with $e_{\mu,m}$,
up to a normalization.

We then connect our results to the Hilbert scheme.
Let $S = \C^2$, and let the one-dimensional torus acts by
\begin{equation}
z\cdot(x,y) = (zx,z^{-1}y),\quad z \in T.
\label{T}
\end{equation}
The vertex operator in this paper equals the operator $\bW$ defined in \cite{CO},
in respective fixed-point bases of the two moduli spaces. We show that this isomorphism is represented by
by a characteristic class on the product of the two moduli spaces,
and that it intertwines the two vertex operators. 

Thus, this paper serves two purposes. First, it connects a family of vertex operators on
the cohomology of the Hilbert scheme to a simpler moduli space.
Second, many difficult identities follow from the equality of an integral over a complex variety with its
localization expression. We recover the aforementioned relationships between infinite-dimensional Lie
algebras and vertex operators as limits of such identities as the variety tends to the infinite Grassmannian.

\subsection{Acknowledgments} The author would like to thank Andrei Okounkov and David Nadler for valuable discussions.
Okounkov suggested connecting the Hilbert scheme to the Sato Grassmannian.

\section{Grassmannians}
\label{grasssection}
Let
\begin{equation}
\cH^{M,N} = \left(x^{M}\cdot \C[[x]]\right)/(x^N),\quad M \leq N,\quad \dim(\cH^{M,N}) = N-M.
\end{equation}
There is an action of the circle on $\cH^{M,N}$ by translation,
\[(z\cdot f)(x) = f(z^{-1}x),\]
and an invariant Hermitian metric,
\[\langle x^m,x^n \rangle = \delta_{m,n}.\]
These spaces have an obvious system of commuting inclusion and projection maps
\[i_{M,M'} : \cH^{M,N} \rightarrow \cH^{M',N} \quad
p_{N',N} : \cH^{M,N'} \rightarrow \cH^{M,N}, \]
for $M' \leq M$, $N \leq N'$. 

Let
\[\cG_{m}^{M,N} = \Gr(m+N,\cH^{M,N}),\]
the Grassmannian of $m+N$-dimensional subspaces of $\cH^{M,N}$, and define
\[f_m^{M,M'} : \cG_{m}^{M,N} \rightarrow \cG_{m}^{M',N},\quad g_m^{N,N'} : \cG_{m}^{M,N} \rightarrow \cG_{n}^{M,N'}\]
by
\[f_{m}^{M,M'}(V) = i_{M,M'}(V),\quad g_m^{N,N'}(V) = p_{N',N}^{-1}(V).\]
Their pullbacks form an inverse system on
\[\V'_{m,M,N} = H^*_T(\cG_{m}^{M,N},\C) \otimes_{\C[t]} \C(t),\quad t \in \Lie(T),\]
with the induced action of $T$ on $\cG_{m}^{M,N}$.
The dual maps form a direct system
\[f^{M,M'}_* : \V_{m}^{M,N} \rightarrow \V_{m}^{M',N},\quad g^{N,N'}_* : \V_{m}^{M,N} \rightarrow \V_{m}^{M,N'}\]
on $\V_{m}^{M,N}$, the dual vector space of $\V'_{m,M,N}$ over $\C(t)$.
Let $\V_m$ be the direct limit over $M,N$, and let $h_{m}^{M,N} : \V_{m}^{M,N} \rightarrow \V_m$ be the induced maps.

By the localization theorem, $v_U = i_U^*$ constitute a basis of $\cV_{m}^{M,N}$,
where $i_U : pt \rightarrow \cG_{m}^{M,N}$ is the inclusion map of a fixed subspace, $U$.
Therefore
\[\cV_m \cong \bigoplus_{\mu} \C(t) \cdot v_{\mu,m},\]
where $\mu$ is a partition,
\[v_{\mu,m} = h_{m}^{M,N}(v_{\mu,m}^{M,N}),\quad 
v_{\mu,m}^{M,N}=v_{U},\quad U = V_{\mu,m}^{M,N},\]
and
\[V_{\mu,m}^{M,N} = \bigoplus_{1 \leq i \leq m+N} \C \cdot x^{-m-(\mu_i-i+1)} \in \cG_{m}^{M,N},\quad
-M,N \gg 0.\]

$\cV'_{m,M,N}$ inherits an inner-product on cohomology,
\[(u,v)_{m}^{M,N} = \int_{\cG_{m}^{M,N}} u\cdot v,\]
which by localization is equal to
\[\sum_{U \in \left(\cG_{m}^{M,N}\right)^T} \frac{v_U\otimes v_U}{\alpha(U,U^\perp)^{-1}}  
\in \V_m^{M,N} \otimes \V_m^{M,N},\quad \alpha(U,V) = e\left(\Hom(U,V)\right).\]
Here $e$ is the Euler class, and $\Hom(U,U^\perp)$ is the tangent
space to $\cG_{m}^{M,N}$ at $U$, viewed as an equivariant bundle over a point.
Explicitly, if $\C_k$ is the representation of $T$ by $z \mapsto z^k$, and
\[U \cong \bigoplus_{k} \C_{k},\quad V \cong \bigoplus_{k'} \C_{k'},\]
then
\[\alpha(U,V) = t^{\dim(U)\dim(V)}\prod_{k,k'} (k'-k) \in H^*_T(pt) \cong \C[t].\]

This inner-product is not compatible with $f_*^{M,M'}$ or $g_*^{N,N'}$, but the normalized inner-product
\[\sum_{U}  \frac{c''_{m,m,M,N}}
{c_{\mu,m,M}\cdot c'_{\mu,m,N}}\frac{v_U \otimes v_U}{\alpha(U,U^\perp)^{-1}},\]
%
%
\[c_{\mu,m,M} = t^{|\mu|}\prod_{(i,j) \in \mu}(-M-m+i-j),\quad
c'_{\nu,n,N} = t^{|\nu|}\prod_{(i,j) \in \nu}(N+n-i+j),\]
%
%
\[c''_{m,n,M,N} = t^{(m+N)(-n-M)}\prod_{i=-(N-1)}^m \prod_{j=n+1}^{-M} (j-i),\]
is compatible, and so defines an element of 
$\cV_{m} \otimes \cV_{m}$. It is given in the dual basis by
\begin{equation}
\label {hook}
(v_{\mu,m},v_{\nu,m}) = \delta_{\mu,\nu} t^{2|\mu|} 
(-1)^{|\mu|}\prod_{\Box \in \mu} h(\Box)^2,
\end{equation}
where $h(\Box)$ is the hook length. 

There is a unique Hermitian inner-product with 
$\langle v_{\mu,m}, v_{\nu,n} \rangle = (v_{\mu,m},v_{\nu,n})$
with respect to the conjugation $\overline{f}(t) = \overline{f(-t)}$. Then
\begin{equation} \label{ebasis}
e_{\mu,m} = t^{-|\mu|}\prod_{\Box \in \mu} h(\Box)^{-1} v_{\mu,m}
\end{equation}
is an orthonormal basis. At $t=\sqrt{-1}$, identifying $e_{\mu,m}$ with the basis \eqref{wedgebasis}
gives an isomorphism $\cV_m \cong \Lambda_m^{\infty/2} (\Z \cdot \C)$.

\section{Vertex Operators}

Consider the flag variety
\begin{equation}
 \F_{m,n}^{M,N} = \left\{(U,V) \in \cG_{m}^{M,N} \times \cG_{n}^{M,N}\big|\ U \subset V\right\},\quad m \leq n.
\end{equation}
$\F_{m,n}^{M,N}$ has a canonical bundle with fiber $V/U$ over each point $(U,V)$,
denoted also by $V/U$, and
an associated determinant line bundle $\cL$. 
%
%
Given a symmetric polynomial, $f$, let
$c_f$ be the characteristic class such that $c$ is a homomorphism, and
\[c_{e_n} = c_n,\]
where $e_n$ is the elementary symmetric polynomial.
Define an inner-product by
\begin{equation}
(c_1,c_2)_{f,z}^{M,N} = \int_{\cF_{m,n}^{M,N}} p_1^*(c_1) p_2^*(c_2) \ch(\cL) c_f(V/U),
\end{equation}
\[c_1 \in \V'_{m,M,N},\quad c_2 \in \V'_{n,M,N},\quad
p_1 \times p_2 : \cF_{m,n}^{M,N} \rightarrow \cG_{m}^{M,N} \times \cG_{n}^{M,N}\]
\newline
for $m \leq n$.
This is an inner-product on $\V'_{m,M,N} \times \V'_{n,M,N}$ with 
values in $\C(z,t)$, where we have substituted $e^t=z$ in the Chern character.
We prefer to think of $z,t$ as independent variables. When $m=n$, this becomes the 
inner-product of the previous section.

Using the formula for the tangent bundle to the flag variety, we obtain the dual inner product in coordinates,
\begin{equation}
\nonumber
\left(v_{U},v_{V}\right)_{f,z}^{M,N} = 
\delta(U \subset V)\alpha(U,V^\perp) c_f(V/U) \ch(\det(V/U)).
\end{equation}
Normalizing as before,
\begin{equation}
\label{fn}
\left(v_{U},v_{V}\right)_{f,z} = \frac{c_{\mu,m,M}c'_{\nu,n,N}}
{c''_{m,n,M,N}}\left(v_{U},v_{V}\right)_{f,z}^{M,N},
\end{equation}
we obtain an inner-product on $\cV_{m} \times \cV_{n}$, given in coordinates by
\begin{equation}
\begin{split}
\left(v_{\mu,m},v_{\nu,n}\right)_{f,z}   & =  z^kt^{|\mu|+|\nu|} \delta(V_{\mu,m} \subset V_{\nu,n}) 
\prod_{\Box \in \mu} (a+a_\nu(\Box)+\ell_\mu(\Box)+1)\cdot \\
 & \prod_{\Box \in \nu} (a-a_\mu(\Box)-\ell_\nu(\Box)-1)c_f(V_{\nu,n}/V_{\mu,m}),
\end{split}
\label{W}
\end{equation}
\[a= n-m, \quad  k = |\nu|+\frac{n(n+1)}{2}-|\mu|-\frac{m(m+1)}{2}.\]
%
%
where $a_{\mu}$, $\ell_\mu$ are the (possibly negative) arm and leg-lengths in $\mu$.

If $n < m$, let
\begin{equation}
\nonumber
\left(u,v\right)_{f,z}^{M,N} = (-1)^{ma}\left(v,u\right)_{f,z^{-1}}^{M,N},\quad
\end{equation}
where $u$ and $v$ are homogeneous classes
of degrees $m$ and $n$, and
\begin{equation}
\left(Y(a,f,\pm z)\cdot u,v\right) = 
z^{-a(a+1)/2} \left(u,v\right)_{f,z}.
\end{equation}
The expression above vanishes if $a \neq n-m$, and $\pm$ means $sgn(a)$.
The signs are designed so that we always arrive at the expression in \eqref{W}, which is defined for all $a$.

The coefficient of $z^k$ of $Y(a,f,z)$ is an honest operator on $\V = \bigoplus_m \V_m$,
though not for any particular value of $z$.
Instead, $Y(a,f,z)$ is a formal power series with coefficients in $\End(\V)$,
\[Y(a,f,z) \in \End(\V)[[z^{\pm 1}]],\quad \V = \bigoplus_{\mu,m} \C(t) \cdot v_{\mu,m},\]
which is sometimes called a field.

\section{Locality}

The commutation relations between the vertex operators
are collectively called locality, and are described
by the supercommutator,
\begin{equation}
\begin{split}
\{& Y(a,f,z),Y(b,g,w)\} = \\ 
Y(a,f,z)& Y(b,g,w)- (-1)^{ab} Y(b,g,w) Y(a,f,z).
\end{split}
\end{equation}
We deduce the commutation relations among the $Y(a,f,z)$ using
the finite-dimensional approximations.
\begin{lemma}
Let $U \in \left(\cG_{l}^{M,N}\right)^T$, $W \in \left(\cG_{n}^{M,N}\right)^T$, 
$f=e_\kappa$, $g=e_{\kappa'}$, where $e_\kappa = \prod_i e_{\kappa_i}$, and let
%
%
%
%
\[A = \sum_{V \in \left(\cG_{m}^{M,N}\right)^T} \frac{(v_U,v_V)_{f,z}^{M,N} (v_W,v_V)_{g,w}^{M,N}}{(v_V,v_V)^{M,N}},\]
\[A' = \sum_{V \in \left(\cG_{m'}^{M,N}\right)^T} \frac{(v_U,v_V)_{g,w}^{M,N} (v_V,v_W)_{f,z}^{M,N}}{(v_V,v_V)^{M,N}},\quad
B = A-(-1)^{ab}A',\]
where $m' = l+n-m$, $a=m-l$, $b=n-m$.
\begin{description}
 \item[a.] If $a,b \geq 0$, or $a,b \leq 0$, then $B = 0$.

\item[b.] If $a>0$, $b<0$, then
\begin{equation}
\label{h}
B = h_1(z,w)(z+w)^{K_1}+ \cdots + h_{d}(z,w)(z+w)^{K_{d}},
\end{equation}
where 
\[d = m-\dim(Y),\quad K_j = j(\dim(Z)+j-2d-\ell(\kappa)-\ell(\kappa')), \vspace{.25cm}\]
\[X = U \cap V,\quad Y = U + V,\quad Z = X \oplus Y^{\perp}.\]
%
%
%
%
%
\end{description}
\end{lemma}

\begin{proof}

Assume for simplicity that $w=-1$.
To prove the first part, suppose $a,b \geq 0$. Then
\[A = \sum_{V \in \left(\cG_{m}^{M,N}\right)^T} \delta(U \subset V)\delta(V \subset W)
\frac{\alpha(U,V^\perp)\alpha(V,W^\perp)}{\alpha(V,V^\perp)}\cdot\]
\[\ch(\det(V/U)) c_f(V/U) c_{g}(W/V) =\]

\[\alpha(U,W^\perp)\int_{\Gr(a,W/U)} \ch(\det(V))
c_f(V)c_{g}(V^\perp),\]
where in the second line, $V$ is the universal bundle on the Grassmannian, and 
$V^\perp$ is perpendicular within $W/U$. Similarly
\[A' = \alpha(U,W^\perp) \int_{\Gr(b,W/U)} \ch(\det(V^\perp))
c_{g}(V)c_{f}(V^\perp).\]

The map 
\[\perp:\Gr(a,W/U) \rightarrow \Gr(b,W/U)\]
is a $T$-equivariant isomorphism of real manifolds of degree $(-1)^{ab}$, which
interchanges the integrands, proving that $A=(-1)^{ab}A'$. The case $a,b < 0$ is similar.

For the second part, let $\dim(X)=p$, $\dim(Y)=q$. 
\[A = (-1)^{la}\sum_{V \in \left(\cG_{m}^{M,N}\right)^T} \ch(\det(V/U))
c_f(V/U)c_{g}(V/W) \cdot\]
\[\delta(V \supset Y)\frac{\alpha(U,V^\perp)\alpha(W,V^\perp)}{\alpha(V,V^\perp)} = \]

\[(-1)^{la+pd}\ch(\det(Y/U))
\alpha(X,Y^\perp)c_f(Y/U) c_g(Y/W)\cdot\]

\[ \sum_{V \in \left(\cG_{m}^{M,N}\right)^T} \delta(V \supset Y) 
\frac{c_{fg}(V/Y)\ch(\det(V/Y))}{\alpha(V/Y,X \oplus V^\perp)} = \]

\[B_0\cdot F(Y^\perp,X,c_{fg},d)\]
where
\begin{equation}
F(Y^\perp,X,c_f,d) = \sum_{V \in \Gr(d,Y^\perp)^T} \frac{\ch(V)c_f(V)}
{ \alpha(V,X\oplus Y^\perp \ominus V)}.
\label{Fterm}
\end{equation}
In the second line, we used the fact that
\[\alpha(X,V/Y) = (-1)^{pd}\alpha(V/Y,X).\]
Doing a similar calculation for $A'$, we find that
\[B = B_0 \cdot H(Y^{\perp},X,c_{fg},d),\]
where
\[H(Y^\perp,X,c_f,d) = F(Y^\perp,X,c_f,d)-(-1)^{d} F(X,Y^\perp,c_f,d).\]
We must prove that this has the form \eqref{h}.

Consider first the base case $X=\{0\}$. Then
\[H(Y^\perp,\{0\},c_f,d) = F(Y^\perp,\{0\},c_f,d) = \frac{h(z)}{t^{K_d}}.\]
%
At $z=e^t$, \eqref{Fterm} represents a localization integral over a compact the Grassmannian,
so it be an element $\C[[t]]$, the closure of $H^*_T(pt)$. 
Therefore $h(z)$ vanishes at $z=1$ to degree $K_d$.
Now suppose $c_f(V) = c_S(V) = \alpha(V,S)$ where $S$ is a representation of the circle.
The remaining cases follow from the recursion relation

\[H(Y^\perp,X,c_S,d)-(-1)^{d} H(Y^\perp \ominus \C_k, X \oplus \C_k,c_S,d) = \]

\begin{equation}
\frac{\alpha(\C_k,S) z^k}{\alpha(X \oplus Y^\perp \ominus \C_k,\C_k)}
H(Y^\perp \ominus \C_k,X,c_{S\oplus \C_k},d-1),
\vspace{.25cm}
\end{equation}
which is straightforward to verify.

\end{proof}

\begin{thm} 
If $a,b\geq0$, or $a,b \leq 0$, then $\{Y(a,f,z),Y(b,g,w)\}=0$.
Otherwise, there exists $K>0$ such that
\[(z-w)^K\{Y(a,f,z),Y(b,g,w)\} = 0.\]%
\label{locality}
\end{thm}

\begin{proof}

For simplicity, we may assume $w=1$.
We substitute $M=-N$, and express the commutator as a
limit as $N$ approaches $\infty$. Since
\[\lim_{N\rightarrow \infty} \frac{c_{\mu,m,N}}{N^{|\mu|}} = 
\lim_{N\rightarrow \infty} \frac{c'_{\mu,m,-N}}{N^{|\mu|}} = 1,\]
the first part follows immediately from the lemma.

For the second, we see that
\[\left(\{Y(a,f,z),Y(b,g,1)\}\cdot v_{\mu,m},v_{\nu,n}\right) = \lim_{N \rightarrow \infty} C_N\cdot H(Y^\perp,X,c_{fg},d),\]
where $X$ and $Y$ are as above, with $U = V^{-N,N}_{\mu,m}$, $W=V^{-N,N}_{\nu,n}$.
This vanishes if $\dim(Y/X) > |a-b|$. It follows from this fact and the lemma
that there are overall bounds on the upper and lower degrees of $z^{jN}h_j(z,1)$ for all $j,\mu,m,\nu,n$.
It is straightforward to check that if
\[h'_{1,N}(z)z^{-N}(1+z)^{2N+c_1} +\cdots +h'_{d,N}(z)z^{-dN}(1+z)^{2dN+c_d}\]
converges to $h'(z)$ as a formal distribution, and the upper and lower degrees of $h'_{j,N}(z)$ are bounded, then
$h'(z)$ is a linear combination of derivatives of the delta function, so that
\[(z-1)^K h(z)' = 0,\]
for some overall $K$, independent of $U,W$.
\end{proof}

\section{Bosonization and the Virasoro Algebra}
We now define the Heisenberg and Virasoro actions on $\cV_m$,
and use the previous section to obtain the intertwining properties of
$Y(a,z) = Y(a,c_0,z)$. As a result, we obtain a simple proof of the boson-fermion correspondence.

The recursion relation shows that when $a=-1$, $b>0$, $f=g=1$, $t=1$,
\begin{equation}
B = \alpha(U,W^\perp)\prod_{k \in W/U}(k+z\partial_z)\cdot \frac{z^Mw^{-(N-1)}(z+w)^{N-M-1}}{(N-M-1)!}.
\label{fermicom}
\end{equation}
The expression on the right hand side comes from the formula
\[H(Y^\perp,X,c_0,1) = H(Y^\perp \oplus X,\{0\},c_0,1) = \int_{\mathbb{P}(Y/X)} \ch(\cL)\]
where $\cL$ is the universal bundle on $\mathbb{P}(Y/X)$.
Taking the limit normalized as in \eqref{fn} at $b=1$, we get the Clifford algebra relations
\begin{equation}
\left\{\psi(z),\psi^*(w)\right\} = \delta(z-w),
\label{clifford}
\end{equation}
where
\[\psi(z) = Y(1,z),\quad \psi^*(z) = Y(-1,z),\quad \delta(z-w) = \sum_j z^jw^{-j-1}.\]

Let
\[\alpha_0 = \sum_{j>0} \psi_j \psi_j^*-\sum_{j \leq 0} \psi^*_j\psi_j,\quad
\alpha_n = \sum_j \psi_{j-n}\psi_{j}^*,\quad n \neq 0\]
where $\psi_j$ and $\psi_j^*$ are defined by
\[\psi(z) = \sum_j \psi_jz^{-j-1},\quad \psi^*(z) = \sum_j \psi^*_j z^{-j}.\]
It is easy to see that these operators act by ``wedging and contracting''
in the basis \eqref{ebasis}, and that
$\alpha_0$ is multiplication by $m$ on $\cV_m$.
By \eqref{clifford}, these operators satisfy the Heisenberg commutation relations
\[[\alpha_i,\alpha_j] = i\delta_{i,-j}\]
We also define the Virasoro generators
\begin{equation} 
L_n = \sum_k k\psi_k \psi^*_{k+n}.
\label{Ln}
\end{equation}

\begin{thm}
With the above notations,
\begin{enumerate}[(1)]
\item \label{bosonization}
\begin{equation} \nonumber
Y(a,z) =
Q^az^{a\alpha_0} \exp\left(a\sum_{n>0} \frac{\alpha_{-n}z^n}{n}\right)
\exp\left(-a\sum_{n>0} \frac{\alpha_{n}z^{-n}}{n}\right),
\end{equation}
where $Q\cdot v_{\mu,m} = v_{\mu,m+1}$.
\item \label{conformal}
\begin{equation}\nonumber
[L_n,Y(a,z)] = \left(z^{n+1}\partial_z+\frac{a(a-1)}{2}(n+1)z^n\right)\cdot Y(a,z).
\end{equation}
\item \label{vertexalgebra}
$\cV$ is a conformal vertex algebra,
isomorphic to the free fermionic vertex superalgebra.
Under this isomorphism, part (\ref{bosonization}) is the boson-fermion correspondence (see \cite{K,FB}).
\end{enumerate}
\end{thm}
\begin{proof}

Let us compute the commutator
\[\left[\alpha(z),Y(a,w)\right] = \psi(z)\left\{\psi^*(z),Y(a,w)\right\}.\]
As in theorem \ref{locality},
we deduce the expression on the right hand side by taking a limit of the matrix elements
\begin{equation}
\begin{split}
\sum_{V_0 \in \cG_{m_0}}\frac{(U,V_0)_z^{M,N}}{(V_0,V_0)^{M,N}} 
& \left(\sum_{V_1 \in \left(\cG_{m_1}^{M,N}\right)^T} 
\frac{(V_0,V_1)^{M,N}_z(V_1,W)^{M,N}_w}{(V_1,V_1)^{M,N}} - \right. \\
(-1)^a & \left. \sum_{V_1'\in\left(\cG_{m_1'}^{M,N}\right)^T}
\frac{(V_0,V_1')^{M,N}_w(V_1',W)^{M,N}_z}{\omega^{M,N}(V_1',V_1')}\right),
\end{split}
\nonumber
\end{equation}
where
\[U \in \left(\cG_{l}^{M,N}\right)^T,\quad W \in \left(\cG_{n}^{M,N}\right)^T,\]
\[m_0 = l+1,\quad m_1 = l,\quad m_1' = n+1,\quad a=n-l.\]
Inserting \eqref{fermicom}, setting $M=-N$, and taking
the normalized limit as $N \rightarrow \infty$ gives
\[\alpha(U,W^\perp)\sum_{V \in \mathbb{P}(W/U)^T} \frac{\ch_z(V)\ch_w(V^\perp)}{\alpha(V,V^\perp)}
\prod_{k' \in V^\perp} (k'+z\partial_z)\cdot \delta(zw^{-1}-1).\]
Extracting the coefficient of $z^n$, we get
\[w^n\alpha(U,W^\perp)\ch_w(W/U)\sum_{k}
\prod_{k' \neq k}\frac{k'-k-n}{k'-k},\]
where $k,k'$ are the torus weights of $W/U$. The sum is the localization expression
for the Euler characteristic of $\mathbb{P}(W/U)$ which equals $a$ for all $n,U,W$. 
The case $a < 0$ is similar with $\cG r(-a-1,W/U)$ replacing projective space.
It follows that
\[\left[\alpha_n,Y(a,z)\right] = az^n Y(a,z).\]
This determines $Y(a,z)$, and in fact it equals the expression in part (\ref{bosonization}).

Part (\ref{conformal}) is similar, but uses the alternate formula
\[\sum_{k} k\prod_{k' \neq k}\frac{k'-k-n}{k'-k} = \int_{\mathbb{P}(W/U)} c_1(\cL) \sum_i c_i(T \mathbb{P}(W/U))(-n)^{a-i} = \]
\[c_1(W/U)+\frac{a(a-1)}{2}n.\]
In particular,
\begin{equation} [T,Y(a,z)] = \partial_z Y(a,z), \label{Lminusone} \end{equation}
where $T = L_{-1}$ is the time translation operator. Therefore part (\ref{vertexalgebra}) follows
from theorem \ref{locality} and the reconstruction theorem \cite{FB}.

\end{proof}

It would be interesting if the additional operators $Y(a,f,z)$ could be included the vertex algebra. 
Unfortunately, while they are mutually local, they do not satisfy \eqref{Lminusone}, and so
the reconstruction theorem does not apply.

\section{The Hilbert Scheme} 

Let $S$ be a smooth surface with an action of a torus, $T$.
Then $\Hilb_k S$ parametrizes subschemes $Z \subset S$ with $\dim_\C H^0(\cO_Z) = n$,
or equivalently, their ideal sheaves
\[I_Z = \ker (\cO_S \rightarrow \cO_Z).\]
By definition, there is a universal subscheme,
\[\fZ \subset \Hilb S \times S,\]
and its corresponding ideal sheaf $\cI$. Let
\[\cU'_S = \bigoplus_k H^*_T(\Hilb_k S) \otimes_{\C[\lt^*]} \C(\lt^*),\]
and let $\cU$ denote its dual vector space over $\C(\lt^*)$.

In \cite{CO}, Okounkov and the author defined a class
\[\bE = {p_{12}}_*\left(\left(
\cO_{\fZ_1}^\vee+\cO_{\fZ_2}-\cO_{\fZ_1}^{\vee}\cdot\cO_{\fZ_2}\right)\cdot
p^*_{3}(\cL)\right) \in K_T(\Hilb_k S \times \Hilb_l S),\]
where $\cL$ is a line bundle on $S$, and $\cO_{\fZ_i} = p_{i3}^* \cO_{\fZ}$.
This is a virtual vector bundle of rank $k+l$, and so we may consider its Euler class,
\[c_{k+l}\left(\bE(\cL)\right) \in H^*_T(\Hilb_k \times \Hilb_l) \subset \cU'_S \otimes \cU_S'.\]
If $S$ has compact fixed loci under $T$, then the vector space
$\cU'_S$ has an inner-product over $\C(\lt^*)$ given as above.
The above characteristic class defines an operator
\[\bW(\cL,z) \in \End(\cU_S')[[z,z^{-1}]],\]
given by%
\begin{equation}
\left(\bW(\cL)\cdot \eta, \xi\right) = z^{l-k}
\int_{\Hilb_k \times \Hilb_l} p_1^*(\eta) \cup p_2^*(\xi)\cup e(\bE)
\label{bW}
\end{equation}
which we gave a formula for in terms of Nakajima's Heisenberg operators,
and the first Chern class of the canonical bundle of $S$.

An important case is $S=\C^2$ with the torus action
\[(z_1,z_2)\cdot (x,y) = (z_1x,z_2y),\]
and where $\cL$ is
a character of $T$ viewed as an equivariant bundle on $\C^2$. 
In this case the fixed-points of the action of $T$ on $\Hilb_k \C^2$ correspond to the monomial ideals
\begin{equation}
I_\mu = (x^{\mu_1},x^{\mu_2}y,x^{\mu_3}y^2,...,y^{\ell{\mu}}),
\end{equation}
where $\mu$ is a partition. By the localization formula,
\[\cU \cong \bigoplus_{\mu} \C(t_1,t_2)\cdot u_\mu,\]
where $u_\mu$ is the pullback of the inclusion of $I_\mu$, and $\cU$ is the dual of $\cU'$ over $\C(t)$.
For a reference on localization and the equivariant cohomology of the Hilbert scheme, see \cite{V,LQW_Jack}.

Let us calculate the matrix elements of $\bW(\cL)$ in this basis.
To do this we need the character of the fiber of the bundle $\bE(\cL)$, over a pair of
fixed-points $(I_\mu,I_\nu) \in \Hilb_k \C^2\times \Hilb_l \C^2$.
Since
\[\ch H^0(\cO_Z) = \ch H^0(\cO_{\C^2})-\ch H^0(I_Z) \in \C(z_1^{-1},z_2^{-1}),\]
we find that
\begin{equation}
\ch \bE(\cL)\big|_{(I_\mu,I_\nu)} = 
\left(\ch\chi(\cO,\cO)- \ch \chi(I_\mu,I_\nu)\right)\ch(\cL),
\label{Ediff}
\end{equation}
where
\[\chi_S(F,G) = \Ext^0_S(F,G)-\Ext^1_S(F,G)+\Ext_S^2(F,G),\]
the derived push-forward of $Hom_S(F,G)$ to a point.
By Riemann-Roch,
\[\ch \chi_{\C^2} (I_\mu,I_\nu) = e(\C^2)\td (\C^2)^{-1}
\ch H^0(I_\mu^\vee)\ch H^0(I_\nu)  = \vspace{.25cm}\]
\begin{equation}
(1-z_1)(1-z_2)\overline{\ch H^0(I_\mu)}\ch H^0(I_\nu)
\in \C((z_1^{-1},z_2^{-1}))
\label{chi}
\end{equation}
where $F^\vee$ is the dual in equivariant $K$-theory, $\td$ is the Todd class,
\[\ch H^0(I_\mu) = \sum_{i,j \geq 1,\ (i,j) \notin \mu} z_1^{1-i}z_2^{1-j},\quad
e(\C^2)\td (\C^2)^{-1} = (1-z_1^{-1})(1-z_2^{-1}),\]
\[\ch(H^0(F^\vee)) = \overline{z_1^{-1}z_2^{-1} \ch(H^0(F))}, \quad \overline{f(z_1,z_2)} = f(z_1^{-1},z_2^{-1}).\]
%
The order of the multiplication matters in \eqref{chi}. For instance, the product of the last two terms
is divergent. On the other hand, multiplication by $(1-z_1)(1-z_2)$
of either term yields an element of $\C(z_1,z_2)$. We have assumed that multiplication happens
from left to right, producing an element of $\C((z_1^{-1},z_2^{-1}))$.

However, choosing a different order of multiplication does not affect the difference \eqref{Ediff},
so long as we choose the order consistently for each term. We may therefore write
\[\ch \bE(\cL)\big|_{(I_\mu,I_\nu)} = 
\left(\overline{\ch H^0(I_{\emptyset})(1-z_2^{-1})}\ch H^0(I_{\emptyset})(1-z_1) \right. -\]
\[\left. \overline{\ch H^0(I_{\mu})(1-z_2^{-1})}\ch H^0(I_{\nu})(1-z_1) \right) \ch \cL =\vspace{.25cm}\]

\[\left(\overline{\sum_{i} z_1^{1-i}z_2^{-\mu^t_i}} \sum_{i} z_1^{1-\nu_i}z_2^{1-i}-
\overline{\sum_{i} z_1^{1-i}} \sum_{i} z_1z_2^{1-i}\right)\ch \cL \vspace{.25cm}=\]
%
\begin{equation}
\left(\sum_{\square\in \mu} z_1^{a_\mu(\square)+1} \, z_2^{-l_\nu(\square)}+
\sum_{\square\in \nu} z_1^{-a_\nu(\square)}\, z_2^{l_\mu(\square)+1}\right)\ch \cL,
\label{Echar}
\end{equation}
where $\mu^t$ is the transposed partition.
Furthermore, the restriction of $\bE = \bE(\cO)$ along the diagonal is the tangent bundle of the Hilbert scheme,
and the above expression reduces to the well-known formulas for its character.
It follows that in the dual coordinates,
\[\left(\bW(\cL) \cdot u_\mu,u_\nu\right) = z^{|\nu|-|\mu|}t^{|\mu|+|\nu|}\prod_{\square \in \mu} 
\left(c_1(\cL)+a_\mu(\square)t_1-\ell_\nu(\square)t_2+t_1\right)\cdot\]
\begin{equation}
\prod_{\square \in \nu} \left(c_1(\cL)-a_\nu(\square)t_1+\ell_\mu(\square)t_2+t_2\right).
\label{Whooks}
\end{equation}

We may now explain how this relates to our results in this paper.
Let $\cU = \cU_{\C^2}$, with the simpler torus action \eqref{T}.
This amounts to setting
\[(t_1,t_2)=(t,-t),\quad c_1(\cL) = at,\quad a \in \Z\] 
in \eqref{Whooks}. Comparing with \eqref{W},
the two operators are identical, with the partitions transposed. 
We now explain why.

Given a vector space $U \subset \C(x,y)$ and an integer $N$, let
\begin{equation}
U_{N} = U \cap \spn\left\{x^iy^j\ |\ i,j < N\right\} \subset \C(x,y).
\end{equation}
For each $k,m$, and large $N$, there is a $T$-equivariant subbundle 
of the trivial bundle $\C(x,y)$ on $\Hilb_k \C^2$ whose fiber is given by
\[\cI_{m,n,N}\big|_Z = (x^my^nI_Z)_N,\quad \dim(\cI_{m,n,N}) = (N-m)(N-n)-k.\]
There are no jumps in dimension by
the construction of the Hilbert scheme on $\Pt$, and the fact that the $N$th filtered subspace of $\C[x,y]$ is
isomorphic to the $N$th graded subspace of $\C[x,y,z]$.

Now consider the bundle on $\Hilb_k \C^2 \times \cG_m^{M,N}$ given by
\begin{equation}
\label{hilb2grassbundle}
\mathcal{E}_m^{M,N} = \cHom_{\C}(\cI_{-m,0,N}/\cI_{-m,1,N},V^\perp),
\end{equation}
which has dimension $(N+m)(N-m)$ for $k,m$, and sufficiently large $-M,N$.
The Euler class of this bundle defines a map $\cU \rightarrow \cV_m^{M,N}$ whose composition with $h_m^{M,N}$ is the isomorphism
\[\varphi_m : \cU\rightarrow \cV_m,\quad \varphi_m(u_\mu) = v_{\mu^t,m}\]
of inner-product spaces over $\C(t)$. Substituting
\[\ch V_{\mu,m} = \sum_i z^{m+\mu_i-i+1},\quad \ch V_{\nu,n}^\perp = \sum_i z^{n-\nu^t_i+i},\quad \ch \cL = z^a\]
in \eqref{Echar}, we have deduced the following theorem:
\begin{thm}
When $S = \C^2$ with the torus action \eqref{T},
\[\bW(\cL,z) = z^{-ma} \varphi_n^{-1} Y(a,z) \varphi_m,\]
where $a=n-m$, and $\cL$ is the trivial bundle with character $z^a$.
\qed
\end{thm}

Note that $z$ is merely a place holder for the number of points
on the Hilbert scheme side, but varies within each Grassmannian.

\end{document}